\documentclass[a4paper,12pt,twoside]{article}

\usepackage{latexsym,epic,eepic,float}
\usepackage{color,colordvi}
\usepackage{shadow}
\usepackage{amsmath,amssymb}
\usepackage{epsfig}
\usepackage{lscape}
\usepackage{rotating}
\usepackage{subfigure}
\usepackage{multicol}

\usepackage[T1]{fontenc}
\usepackage[utf8]{inputenc}
\usepackage{authblk}

\newtheorem{theorem}{Theorem}[section]

\newtheorem{lem}[theorem]{Lemma}
\newtheorem{coro}[theorem]{Corollary}
\newtheorem{prop}[theorem]{Proposition}

\newcommand{\exit}{{\mbox{\, \vspace{3mm}}} \hfill\mbox{$\square$}}

\rm 

\numberwithin{equation}{section}

\title{Parisian excursion below a fixed level from the last record maximum of L\'evy insurance\\ risk process}

\author{B.A. Surya\footnote{Email address: budhi.surya@vuw.ac.nz; Postal address: School of Mathematics and Statistics, Victoria University of Wellington, Gate 6 Kelburn PDE, Wellington 6140, New Zealand.}\\ Victoria University of Wellington \\ School of Mathematics and Statistics \\ Wellington, New Zealand }

\date{15 February 2018}

\begin{document}
\maketitle \pagestyle{myheadings} \markboth{B.A. Surya} {Parisian excursion from the last record maximum of L\'evy process}
\begin{abstract}
This paper presents some new results on Parisian ruin under L\'evy insurance risk process, where ruin occurs when the process has gone below a fixed level from the last record maximum, also known as the high-water mark or drawdown, for a fixed consecutive periods of time. The law of ruin-time and the position at ruin is given in terms of their joint Laplace transforms. Identities are presented semi-explicitly in terms of the scale function and the law of the L\'evy process. They are established using recent developments on fluctuation theory of drawdown of spectrally negative L\'evy process. In contrast to the Parisian ruin of L\'evy process below a fixed level, ruin under drawdown occurs in finite time with probability one. 

\medskip

\end{abstract}

\section{Introduction}
\label{sec:1}
Let $X=\{X_t:t\geq 0\}$ be a spectrally negative L\'evy process defined on filtered probability space $(\Omega, \mathcal{F}, \{\mathcal{F}_t:t\geq 0\},\mathbb{P})$, where $\mathcal{F}_t$ is the natural filtration of $X$ satisfying the usual assumptions of right-continuity and completeness. We denote by $\{\mathbb{P}_x, x\in\mathbb{R}\}$ the family of probability measure corresponding to a translation of $X$ s.t. $X_0=x$, with $\mathbb{P}=\mathbb{P}_0$, and define $\overline{X}_t=\sup_{0\leq s\leq t} X_s$ the running maximum of $X$ up to time $t$. The L\'evy-It\^o sample paths decomposition of the L\'evy process is given by
\begin{align}\label{eq:LevyIto}
X_t=\mu t + \sigma B_t &+\int_0^t\int\limits_{\{x<-1\}} x\nu(dx,ds) + \int_0^t\int\limits_{\{-1\leq x<0\}} x\big(\nu(dx,ds)-\Pi(dx)ds\big),
\end{align}
where $\mu\in\mathbb{R}$, $\sigma\geq0$ and $(B_t)_{t\geq0}$ is standard Brownian motion, whilst $\nu(dx,dt)$ denotes the Poisson random measure associated with the jumps process $\Delta X_t:=X_t-X_{t-}$ of $X$. This Poisson random measure has compensator given by $\Pi(dx)dt$, where $\Pi$ is the L\'evy measure satisfying the integrability condition:
\begin{equation}\label{eq:integrability}
\int_{-\infty}^0 (1\wedge x^2)\Pi(dx)<\infty.
\end{equation}
We refer to Chapter 2 of \cite{Kyprianou} for more details on paths decomposition of $X$.

Due to the absence of positive jumps, it is therefore sensible to define
\begin{equation}\label{eq:exponent}
\psi(\lambda)=\frac{1}{t}\log\mathbb{E}\big\{e^{\lambda
X_{t}}\big\}=\mu\lambda
+\frac{1}{2}\sigma^{2}\lambda^{2}+\int_{(-\infty,0)}\big(e^{\lambda
x}-1-\lambda x\mathbf{1}_{\{x>-1\}}\big)\Pi(dx),
\end{equation}
which is analytic on ($\mathfrak{Im}(\lambda)\leq 0$). It is easily
shown that $\psi$ is zero at the origin, tends to infinity at
infinity and is strictly convex. We denote by
$\Phi:[0,\infty)\rightarrow [0,\infty)$ the right continuous inverse
of the Laplace exponent $\psi(\lambda)$, so that
\begin{equation*}
\Phi(\theta)=\sup\{p>0:\psi(p)=\theta\} \quad \textrm{and} \quad
\psi(\Phi(\lambda))=\lambda \quad \text{for all} \quad \lambda \geq
0.
\end{equation*}

It is worth mentioning that under the Esscher transform $\mathbb{P}^{\nu}$ defined by
\begin{equation}\label{eq:esscher}
\frac{d\mathbb{P}^{\nu}}{d\mathbb{P}}\Big\vert_{\mathcal{F}_t}=e^{\nu
X_t -\psi(\nu) t} \quad \textrm{for all $\nu\geq 0$,}
\end{equation}
the L\'evy process $(X,\mathbb{P}^{\nu})$ is still a spectrally
negative L\'evy process. The Laplace exponent of $X$ under the new
measure $\mathbb{P}^{\nu}$ has changed to $\psi_{\nu}(\lambda)$
given by
\begin{equation}\label{eq:newLExp}
\begin{split}
\psi_{\nu}(\lambda) \;=&\; \psi(\lambda+\nu)-\psi(\nu), \quad \textrm{for $\lambda\geq -\nu$}. 
\end{split}
\end{equation}
Subsequently, we define by $\Phi_{\nu}(\theta)$ the largest root of equation
$\psi_{\nu}(\lambda)=\theta$ satisfying
\begin{equation*}
\Phi_{\nu}(\theta)=\Phi(\theta +\psi(\nu))-\nu.
\end{equation*}

Furthermore, assume that from some reference point of time in the past $X$ has achieved maximum $y>0$. Define drawdown process $Y=\{Y_t:t\geq 0\}$ of $X$ by
\begin{equation}\label{eq:drawdown}
Y_t=\overline{X}_t \vee y -X_t,
\end{equation}
under measure $\mathbb{P}_{y,x}$. Notice that we altered slightly our notation for the probability measure $\mathbb{P}_{y,x}$ to denote the law of $X$ under which at time zero $X$ has current maximum $y\geq x$ and position $x\in\mathbb{R}$. We simply write $\mathbb{P}_{\vert y}:=\mathbb{P}_{y,0}$ the law of $Y$ under which $Y_0=y$, and use the notation $\mathbb{E}_x$, $\mathbb{E}_{y,x}$ and $\mathbb{E}_{\vert y}$ to define the corresponding expectation operator to the above probability measures. Subsequently, we denote by $\mathbb{E}_{y,x}^{\nu}$ the expectation under $\mathbb{P}_{y,x}^{\nu}$ by which the L\'evy process $X$ has the Laplace exponent $\psi_{\nu}(\lambda)$ (\ref{eq:newLExp}). Recall that since $X$ is a L\'evy process, it follows that $Y$ is strong Markov. 

In recent developments, some results regarding excursion below a (fixed) default level, say zero, of the L\'evy process $X$ with fixed duration (Parisian ruin) have been obtained and applied in finance and insurance (e.g. option pricing, corporate finance, optimal dividend, etc). We refer among others to Chesney et al. \cite{Chesney}, Francois and Morellec \cite{Francois}, Broadie et al. \cite{Broadie}, Dassios and Wu \cite{Dassios2010}, Loeffen et al. (\cite{Loeffen} \& \cite{Loeffen2017}), Czarna and Palmowski \cite{Czarna} and Landriault et al. \cite{Landriault} and the literature therein for further discussions. In these papers, the excursion takes effect from the first time $T_0^-=\inf\{t>0: X_t<0\}$ the process $X$ has gone below zero under measure $\mathbb{P}_x$, and default is announced at the first time $\tau_r=\inf\big\{t>r: \big(t- \sup\{s<t: X_s>0\}\big)>r\big\}$ the L\'evy process has gone below zero for $r>0$ consecutive periods of time. 

In the past decades attention has been paid to find risk protection mechanism against certain financial assets' outperformance over their last record maximum, also referred to as high-water mark or drawdown, which in practice may affect towards fund managers' compensation. See, among others, Agarwal et al. \cite{Agarwal} and Goetzmann et al. \cite{Goetzmann} for details. Such risk may be protected against using an insurance contract. In their recent works, Zhang et al. \cite{Zhang}, Palmowski and Tumilewicz \cite{Palmowski2018} discussed fair valuation and design of such insurance contract. 

Motivated by the above works, we consider a Parisian ruin problem, where ruin occurs when the L\'evy risk process $X$ has gone below a fixed level $a>0$ from its last record maximum (running maximum) $\overline{X}_t\vee y$ for a fixed consecutive periods of time $r\geq0$. This excursion takes effects from the first time $\tau_a^+=\inf\{t>0: \overline{X}_t\vee y - a > X_t\}$ the process under $\mathbb{P}_{y,x}$ has gone below a fixed level $a>0$ from the last record maximum $\overline{X}_t\vee y$. Equivalently, this stopping time can be written in terms of the first passage above level $a>0$ of the drawdown process $Y$ as $\tau_a^+=\inf\{t>0: Y_t > a\}$. Ruin is declared at the first time the process $Y$ has undertaken an excursion above level $a$ for $r$ consecutive periods of time before getting down again below $a$, i.e.,
\begin{equation}\label{eq:ruin-time}
\tau_r=\inf\{t>r: (t-g_t)\geq r \} \quad \textrm{with}\quad g_t=\sup\{0\leq s\leq t: Y_s\leq a\}.
\end{equation}

Working with the stopping time $\tau_r$ (\ref{eq:ruin-time}), we consider the Laplace transforms
\begin{equation}\label{eq:LT}
\mathbb{E}_{y,x}\big\{e^{-u\tau_r}\mathbf{1}_{\{\tau_r<\infty\}}\big\} \quad \textrm{and} \quad
\mathbb{E}_{y,x}\big\{e^{-u\tau_r + \nu X_{\tau_r}}\mathbf{1}_{\{\tau_r<\infty\}}\big\}, 
\end{equation}
for $u,\nu,r\geq 0$ and $y\geq x $. The first quantity gives the law of the ruin time $\tau_r$, whereas the second describes the joint law of the ruin time $\tau_r$ and the position at ruin $X_{\tau_r}$. 
 
The rest of this paper is organized as follows. Section \ref{sec:main} presents the main results of this paper. Some preliminary results are presented in Section \ref{sec:Pre}. Section \ref{sec:proofs} discusses the proofs of the main results. Section \ref{sec:conclusions} concludes this paper.

\section{Main results}\label{sec:main}

The results are expressed in terms of the scale function $W^{(u)}(x)$ of $X$ defined by
\begin{equation}\label{eq:scale}
\int_0^{\infty} e^{-\lambda x} W^{(u)}(x) dx =\frac{1}{\psi(\lambda)-u}, \quad \textrm{for $\lambda>\Phi(u)$,}
\end{equation}
with $W^{(u)}(x)=0$ for $x<0$. We refer to $W_{\nu}^{(u)}$ the scale function under $\mathbb{P}^{\nu}$. Following (\ref{eq:scale}), it is straightforward to check under the new measure $\mathbb{P}^{\nu}$ that
\begin{equation}\label{eq:scale2}
W_{\nu}^{(u)}(x)=e^{-\nu x} W^{(u+\psi(\nu))}(x),
\end{equation}
for all $u$ and $\nu$ such that $u\geq -\psi(\nu)$ and $\psi(\nu)<\infty$. To see this, take Laplace transforms on both sides. We will also use the notation $\overline{W}_{\nu}^{(u)}(x)$ to denote $\int_0^x W_{\nu}^{(u)}(y)dy$.  

It is known following \cite{Chan} that, for any $u\geq 0$, the $u-$scale function $W^{(u)}$ is $C^1(0,\infty)$ if the L\'evy measure $\Pi$ does not have atoms and is $C^2(0,\infty)$ if $\sigma>0$.
For further details on spectrally negative L\'evy process, we refer
to Chapter VI of Bertoin \cite{Bertoin} and Chapter 8 of Kyprianou
\cite{Kyprianou}. Some examples of L\'evy processes for which $W^{(q)}$
are available in explicit form are given by Kuznetzov et al.
\cite{Kuznetzov}. In any case, it can be computed by numerically
inverting (\ref{eq:scale}), see e.g. Surya \cite{Surya}.

In the sequel below, we will use the notation $\Omega_{\epsilon}^{(u)}(x,t)$ defined by
\begin{equation*}
\begin{split}
\Omega_{\epsilon}^{(u)}(x,t)=&\int_{\epsilon}^{\infty} W^{(u)}(z+x-\epsilon)\frac{z}{t}\mathbb{P}\{X_t\in dz\}, \quad \textrm{for $\epsilon\geq 0$},
\end{split}
\end{equation*}
and define its partial derivative w.r.t $x$, $\frac{\partial}{\partial x} \Omega_{\epsilon}^{(u)}(x,t)$, by $\Lambda_{\epsilon}^{(u)}(x,t)$, i.e., 
\begin{equation*}
\begin{split}
\Lambda_{\epsilon}^{(u)}(x,t)=&\int_{\epsilon}^{\infty} W^{(u)\prime}(z+x-\epsilon)\frac{z}{t}\mathbb{P}\{X_t\in dz\}.
\end{split}
\end{equation*}
For convenience, we write $\Omega^{(u)}(x,t)=\Omega_{0}^{(u)}(x,t)$ and $\Lambda^{(u)}(x,t)=\Lambda_{0}^{(u)}(x,t)$.

We denote by $\Omega_{\nu}^{(u)}$ the role of $\Omega^{(u)}$ under change of measure $\mathbb{P}^{\nu}$, i.e.,
\begin{equation}\label{eq:Lambdaqnew}
\Omega_{\nu}^{(u)}(x,t):=\int_{0}^{\infty}W_{\nu}^{(u)}(z+x)\frac{z}{t}\mathbb{P}^{\nu}\{X_t\in dz\},
\end{equation}
similarly defined for $\Lambda_{\nu}^{(u)}(x,t).$ Using (\ref{eq:esscher}), we can rewrite $\Omega_{\nu}^{(u)}(x,t)$ as follows
\begin{equation}\label{eq:OmegaCM}
\Omega_{\nu}^{(u)}(x,t)=e^{-\nu x} e^{-\psi(\nu)t} \Omega^{(u+\psi(\nu))}(x,t).
\end{equation}

The main result concerning the Laplace transform (\ref{eq:LT}) is given below.

\begin{theorem}\label{theo:main}
Define $z=y-x$, with $y\geq x$. For $a>0$ and $u,r\geq 0$, the Laplace transform of $\tau_r$ is given by
 \begin{align}
 \mathbb{E}_{y,x}\big\{e^{-u\tau_r}\mathbf{1}_{\{\tau_r<\infty\}}\big\}&=e^{-ur}\Big\{ 1+ u \Big[\overline{W}^{(u)}(a-z) -\frac{\Omega^{(u)}(a-z,r)}{\Lambda^{(u)}(a,r)}W^{(u)}(a) \nonumber\\
&\hspace{-1cm}+ \int_0^r \Big(\Omega^{(u)}(a-z,t)-\frac{\Omega^{(u)}(a-z,r)}{\Lambda^{(u)}(a,r)}\Lambda^{(u)}(a,t)\Big) dt \Big]  \Big\}. \label{eq:main}
 \end{align}
\end{theorem}

By inserting $u=0$ in (\ref{eq:main}), we see that in contrary to the Parisian ruin probability under the L\'evy process $X$, see e.g. \cite{Loeffen}, we have the following result.
\begin{coro}
For $y\geq x$ and $r\geq 0$, $\mathbb{P}_{y,x}\{\tau_r<\infty\}=1.$
\end{coro}

Following the result of Theorem \ref{theo:main} and applying Esscher transform of measure, the joint law of ruin-time $\tau_r$ and the position at ruin $X_{\tau_r}$ is given below.

\begin{prop}\label{prop:maincor}
Define $z=y-x$, with $y\geq x$, and $p=u-\psi(\nu)$, with $u\geq 0$ and $\nu$ such that $\psi(\nu)<\infty$. For $a>0$ and $r\geq 0$, the joint Laplace transform of $\tau_r$ and $X_{\tau_r}$ is given by
 \begin{align}
 \mathbb{E}_{y,x}\big\{e^{-u\tau_r + \nu X_{\tau_r}}\mathbf{1}_{\{\tau_r<\infty\}}\big\}&=e^{-pr}e^{\nu x}\Big\{ 1+ p \Big[\overline{W}_{\nu}^{(p)}(a-z) -\frac{\Omega_{\nu}^{(p)}(a-z,r)}{\Lambda_{\nu}^{(p)}(a,r)}W_{\nu}^{(p)}(a) \nonumber\\
&\hspace{-0.5cm}+ \int_0^r \Big(\Omega_{\nu}^{(p)}(a-z,t)-\frac{\Omega_{\nu}^{(p)}(a-z,r)}{\Lambda_{\nu}^{(p)}(a,r)}\Lambda_{\nu}^{(p)}(a,t)\Big) dt \Big]  \Big\}. \label{eq:main2}
 \end{align}
\end{prop}

\section{Preliminaries}\label{sec:Pre}
Before we prove the main results, we devote this section to some preliminary results required to establish (\ref{eq:main})-(\ref{eq:main2}); in particular, Theorem \ref{theo:main} on the Laplace transform of $\tau_r$. By spatial homogeneity of the sample paths of $X$, we establish Theorem \ref{theo:main} under the measure $\mathbb{P}_{\vert y}$. To begin with, we define for $a>0$ stopping times:
\begin{equation}\label{eq:exit-time}
\tau_a^+=\inf\{t>0: Y_t>a\} \quad \textrm{and} \quad
\tau_a^-=\inf\{t>0: Y_t<a\} \quad \textrm{under \; $\mathbb{P}_{\vert y}$}.
\end{equation}
Due to the absence of positive jumps, we have by the strong Markov
property of $X$ that $\tau_a^-$ can equivalently be rewritten as
$\tau_a^-=\inf\{t>0: Y_t\leq a\}$ and that
\begin{equation}\label{eq:fptabove}
\mathbb{E}_{\vert y}\big\{e^{-\theta \tau_a^-}\big\}=e^{-\Phi(\theta)(y-a)}.
\end{equation}
This is due to the fact that $\tau_a^-<\tau_{\{0\}}$ a.s., with $\tau_{\{0\}}=\inf\{t>0:Y_t=0\}$, and that $\{Y_t,t\leq \tau_{\{0\}},\mathbb{P}_{y,x}\}=\{-X_t,t\leq T_0^+,\mathbb{P}_{-z}\}$, with $z=y-x$, where $T_a^+=\inf\{t>0:X_t\geq a\}$, $a\geq0$. We refer to Avram et al. \cite{Avram} and Mijatovi\'c and Pistorius \cite{Mijatovic}.

In the derivation of the main results (\ref{eq:main})-(\ref{eq:main2}), we will also frequently apply Kendall's identity (see e.g. Corollary VII.3 in \cite{Bertoin}), which relates the distribution $\mathbb{P}\{X_t\in dx\}$ of a spectrally negative
L\'evy process $X$ to the distribution $\mathbb{P}\{T_x^+\in dt\}$ of its first passage time $T_x^+$ above $x>0$ under $\mathbb{P}$. This identity is given
by
\begin{equation}\label{eq:kendall}
t\mathbb{P}\{T_x^+\in dt\}dx=x\mathbb{P}\{X_t\in dx\}dt.
\end{equation}

To establish our main results, we need to recall the following identities.
\begin{lem}\label{lem:identity1}
Define $s=y-x$, with $y\geq x$. For $a>0$, $u\geq 0$ and $\nu$ such that $\psi(\nu)<\infty$, the joint Laplace transform of $\tau_a^+$ and $Y_{\tau_a^+}$ is given by
\begin{align}
\mathbb{E}_{y,x}\big\{e^{-u\tau_a^+ -\nu Y_{\tau_a^+}}\mathbf{1}_{\{\tau_a^+<\infty\}}\big\}&=(\psi(\nu)-u)e^{-\nu s}\int_{a-s}^{\infty} e^{-\nu z} W^{(u)}(z)dz \label{eq:avram}\\
&\hspace{-3cm} + \frac{W^{(u)}(a-s)}{W^{(u)\prime}(a)}\Big[(\psi(\nu)-u)e^{-\nu a} W^{(u)}(a) -\nu (\psi(\nu)-u)\int_a^{\infty} e^{-\nu z} W^{(u)}(z) dz\Big]. \nonumber
\end{align}
\end{lem}
The identity (\ref{eq:avram}) is due to Theorem 1 in Avram et al. \cite{Avram} taking account of (\ref{eq:scale})-(\ref{eq:scale2}).

\begin{coro}\label{cor:corid1}
Define $s=y-x$, with $y\geq x$. For $a>0$ and $u,\theta\geq 0$,   
\begin{align}
\mathbb{E}_{y,x}\big\{e^{-u\tau_a^+ -\Phi(\theta) Y_{\tau_a^+}}\mathbf{1}_{\{\tau_a^+<\infty\}}\big\}
&=(\theta-u)e^{-\Phi(\theta)s}\int_{a-s}^{\infty}e^{-\Phi(\theta) z}W^{(u)}(z) dz  \label{eq:identity1}\\
&\hspace{-4cm}-\frac{W^{(u)}(a-s)}{W^{(u)\prime}(a)}\Big[(u-\theta)e^{-\Phi(\theta)a}W^{(u)}(a)-(u-\theta)\Phi(\theta)\int_a^{\infty}e^{-\Phi(\theta)z}W^{(u)}(z) dz\Big]. \nonumber
\end{align}
\end{coro}
\begin{proof}
The result follows from inserting $\nu=\Phi(\theta)$ in eqn. (\ref{eq:avram}) and taking account that $\psi(\Phi(\theta))=\theta$, and $\int_0^x e^{-\nu z} W^{(u)}(z)dz=\frac{1}{(\psi(\nu)-u)}-\int_x^{\infty} e^{-\nu z} W^{(u)}(z) dz. \; \;\exit\\$
\end{proof}

Along with Lemma \ref{lem:identity1} and Corollary \ref{cor:corid1}, the three results below are used when applying inverse Laplace transforms to get the main results (\ref{eq:main})-(\ref{eq:main2}).
\begin{lem}\label{lem:identity2}
For a given $\theta > 0$ and $\alpha$ such that $\alpha < \Phi(\theta)$, we have for $y\in\mathbb{R}$,
\begin{align}
\int_0^{\infty}e^{-\theta t}e^{-\alpha y}\int_y^{\infty}e^{\alpha
z}\frac{z}{t}\mathbb{P}\{X_t\in dz\}dt &=\frac{e^{-\Phi(\theta)y}}{\big(\Phi(\theta)-\alpha\big)}.\label{eq:pers1}\\
\int_0^{\infty}e^{-\theta t}e^{-\alpha y}\int_y^{\infty}e^{\alpha z}\int_0^t \frac{z}{u}\mathbb{P}\{X_u \in dz\}du dt &=\frac{e^{-\Phi(\theta)y}}{\theta\big(\Phi(\theta)-\alpha\big)}.\label{eq:pers2}\\
\int_0^{\infty}W^{(u)}(z)\frac{z}{t}\mathbb{P}\{X_t\in dz\}&=e^{ut}, \quad \textrm{for $u\geq0$ and $t>0$}.\label{eq:pers3}
\end{align}
\end{lem}

The results above are slightly generalizations of those given in \cite{Loeffen} and can be proved in similar fashion of \cite{Loeffen} using Kendall's identity (\ref{eq:kendall}) and Tonelli.

\section{Proof of the main results}\label{sec:proofs}

\subsection{Proof of Theorem \ref{theo:main}} 
The proof is established for the case where $X$ has paths of bounded
and unbounded variation. To deal with unbounded variation case, we
will use a limiting argument similar to the one employed in
\cite{Loeffen}, \cite{Loeffen2017} and adjust the ruin time (\ref{eq:ruin-time})
accordingly. For this reason, we introduce for $\epsilon\geq 0$ the
stopping time $\tau_r^{\epsilon}$ defined by
\begin{equation*}
\tau_r^{\epsilon}=\inf\{t>r:\big(t-g_t^{\epsilon}\big)\geq
r\} \; \; \textrm{with} \;\; g_t^{\epsilon}:=\sup\{s<t: Y_s\leq
a-\epsilon \}.
\end{equation*}

This stopping time represents the first time that the L\'evy insurance risk process $X$ has spent a fixed $r>0$ units of time consecutively below pre-specified level $a>0$ from its running maximum $\overline{X}_t\vee y$ ending before $X$ getting back up again to a level $a-\epsilon\geq 0$ below the running maximum. Note that $\tau_r=\tau_r^{0}$. 

By spatial homogeneity of $X$, the proof is given under measure $\mathbb{P}_{\vert y}$ by which $X$ starts at point zero and has current maximum $y$. We have for any $y>a$ that
\begin{equation*}
\begin{split}
\mathbb{E}_{\vert y}\big\{e^{-u\tau_r^{\epsilon}}\mathbf{1}_{\{\tau_r^{\epsilon}<\infty\}}\big\}=
e^{-ur} \mathbb{P}_{\vert y}\{\tau_{a-\epsilon}^->r\} + \mathbb{E}_{\vert y}\big\{e^{-u\tau_r^{\epsilon}}\mathbf{1}_{\{\tau_r^{\epsilon}<\infty, \tau_{a-\epsilon}^- \leq r\}}\big\}.
\end{split}
\end{equation*}
By the strong Markov property of the drawdown process $Y$ (\ref{eq:drawdown}), the second expectation can be worked out using tower property of conditional expectation,
\begin{align*}
\mathbb{E}_{\vert y}\big\{e^{-u\tau_r^{\epsilon}}\mathbf{1}_{\{\tau_r^{\epsilon}<\infty, \tau_{a-\epsilon}^- \leq r\}}\big\}&=\mathbb{E}_{\vert y}\Big\{ \mathbb{E}\big\{e^{-u\tau_r^{\epsilon}}\mathbf{1}_{\{\tau_r^{\epsilon}<\infty, \tau_{a-\epsilon}^- \leq r\}} \big\vert \mathcal{F}_{\tau_{a-\epsilon}^-} \big\}\Big\}\\
&=\mathbb{E}_{\vert y}\Big\{ e^{-u \tau_{a-\epsilon}^-}\mathbf{1}_{\{\tau_{a-\epsilon}^- \leq r\}}
\mathbb{E}_{\vert Y_{\tau_{a-\epsilon}^-}}\big\{e^{-u\tau_r^{\epsilon}}\mathbf{1}_{\{\tau_r^{\epsilon}<\infty\}}\big\}\Big\}\\
&= \mathbb{E}_{\vert y}\big\{ e^{-u \tau_{a-\epsilon}^-}\mathbf{1}_{\{\tau_{a-\epsilon}^- \leq r\}}\big\}
\mathbb{E}_{\vert a-\epsilon}\big\{e^{-u\tau_r^{\epsilon}}\mathbf{1}_{\{\tau_r^{\epsilon}<\infty\}}\big\},
\end{align*}
where the last equality is due to the absence of positive jumps of $X$. Hence,
\begin{equation}\label{eq:mainyneg}
\begin{split}
\mathbb{E}_{\vert y}\big\{e^{-u\tau_r^{\epsilon}}\mathbf{1}_{\{\tau_r^{\epsilon}<\infty\}}\big\}&=
e^{-ur}\big(1- \mathbb{P}_{\vert y}\{\tau_{a-\epsilon}^- \leq r\} \big) \\ & +\mathbb{E}_{\vert y}\big\{ e^{-u \tau_{a-\epsilon}^-}\mathbf{1}_{\{\tau_{a-\epsilon}^- \leq r\}}\big\}
\mathbb{E}_{\vert a-\epsilon}\big\{e^{-u\tau_r^{\epsilon}}\mathbf{1}_{\{\tau_r^{\epsilon}<\infty\}}\big\}.
\end{split}
\end{equation}
Following the above, for $y\leq a$ we have by strong Markov property of $Y$ that 
\begin{align}
\mathbb{E}_{\vert y}\big\{e^{-u\tau_r^{\epsilon}}\mathbf{1}_{\{\tau_r^{\epsilon}<\infty\}}\big\}
&=\mathbb{E}_{\vert y}\Big\{ \mathbb{E}\big\{ e^{-u\tau_r^{\epsilon}}\mathbf{1}_{\{\tau_r^{\epsilon}<\infty\}} \big \vert \mathcal{F}_{\tau_a^+}\big\} \Big\} \nonumber \\
&\hspace{-2.25cm}=\mathbb{E}_{\vert y}\Big\{e^{-u\tau_a^+}\mathbf{1}_{\{\tau_a^+<\infty\}}\Big(e^{-ur}\big(1-\mathbb{P}_{\vert Y_{\tau_a^+}}\{\tau_{a-\epsilon}^- \leq r\}\big)\Big)\Big\}   \nonumber\\
&\hspace{-2cm}+\mathbb{E}_{\vert y}\Big\{e^{-u\tau_a^+}\mathbf{1}_{\{\tau_a^+<\infty\}}\mathbb{E}_{\vert Y_{\tau_a^+}}\big\{e^{-u\tau_{a-\epsilon}^-}\mathbf{1}_{\{\tau_{a-\epsilon}^-\leq r\}}\big\}\Big\}\mathbb{E}_{\vert a-\epsilon}\big\{e^{-u\tau_r^{\epsilon}}\mathbf{1}_{\{\tau_r^{\epsilon}<\infty\}}\big\}  \nonumber \\
&\hspace{-2.25cm}=e^{-ur}\mathbb{E}_{\vert y}\big\{e^{-u\tau_a^+}\mathbf{1}_{\{\tau_a^+<\infty\}}\big\}
-e^{-ur} \mathbb{E}_{\vert y}\Big\{e^{-u\tau_a^+}\mathbf{1}_{\{\tau_a^+<\infty\}}\mathbb{P}_{\vert Y_{\tau_a^+}}\{\tau_{a-\epsilon}^- \leq r\}\Big\} \nonumber\\
&\hspace{-2cm}+\mathbb{E}_{\vert y}\Big\{e^{-u\tau_a^+}\mathbf{1}_{\{\tau_a^+<\infty\}}\mathbb{E}_{\vert Y_{\tau_a^+}}\big\{e^{-u\tau_{a-\epsilon}^-}\mathbf{1}_{\{\tau_{a-\epsilon}^-\leq r\}}\big\}\Big\}\mathbb{E}_{\vert a-\epsilon}\big\{e^{-u\tau_r^{\epsilon}}\mathbf{1}_{\{\tau_r^{\epsilon}<\infty\}}\big\}. 
\label{eq:derivation1}
\end{align}
The first expectation in the last equality of (\ref{eq:derivation1}) can
be worked out in terms of the scale function $W^{(u)}(x)$ using identity (\ref{eq:avram}), whereas the second and the third expectations are given by the following propositions. To establish the results, we denote throughout by $\mathbf{e}_{\theta}$ exponential random time with parameter $\theta$, independent of $X$.

\begin{prop}
For given $u,r,\epsilon \geq 0$ and $a>0$, we have for any $y\geq0$ that
\begin{align}
\mathbb{E}_{\vert y}\Big\{e^{-u\tau_a^+} \mathbf{1}_{\{\tau_a^+<\infty\}} \mathbb{P}_{\vert Y_{\tau_a^+}}\big\{\tau_{a-\epsilon}^- \leq r\big\}\Big\}&=\Omega_{\epsilon}^{(u)}(a-y,r)-u\int_0^r \Omega_{\epsilon}^{(u)}(a-y,t)dt  \nonumber\\
&\hspace{-1.5cm}-\frac{W^{(u)}(a-y)}{W^{(u)\prime}(a)} \Big(\Lambda_{\epsilon}^{(u)}(a,r)-u\int_0^r \Lambda_{\epsilon}^{(u)}(a,t)dt\Big).  \label{eq:prop1}
\end{align}
\end{prop}

\begin{proof}
On recalling (\ref{eq:fptabove}), we have by Tonelli, Lemma \ref{lem:identity1} and Corollary \ref{cor:corid1},
\begin{equation}\label{eq:proof1}
\begin{split}
&\int_0^{\infty} dr e^{-\theta r} \mathbb{E}_{\vert y}\Big\{ e^{-u\tau_a^+}\mathbf{1}_{\{\tau_a^+<\infty\}}\mathbb{P}_{\vert Y_{\tau_a^+}}\big\{\tau_{a-\epsilon}^-\leq r\big\}\Big\}\\
&\hspace{3.5cm}= \frac{1}{\theta}\mathbb{E}_{\vert y}\Big\{ e^{-u\tau_a^+}\mathbf{1}_{\{\tau_a^+<\infty\}}\mathbb{P}_{\vert Y_{\tau_a^+}}\big\{\mathbf{e}_{\theta}\geq \tau_{a-\epsilon}^-\big\}\Big\}\\
&\hspace{3.5cm}=
\frac{1}{\theta} e^{\Phi(\theta)(a-\epsilon)} \mathbb{E}_{\vert y}\Big\{ e^{-u\tau_a^+ - \Phi(\theta) Y_{\tau_a^+}}\mathbf{1}_{\{\tau_a^+<\infty\}}\Big\}.\\
\end{split}
\end{equation}
Furthermore, observe following the result of Corollary \ref{cor:corid1} that for $\theta>u$ we have 
\begin{equation*}
\begin{split}
\mathbb{E}_{\vert y}\Big\{ e^{-u\tau_a^+ - \Phi(\theta) Y_{\tau_a^+}}\mathbf{1}_{\{\tau_a^+<\infty\}}\Big\}
=&\frac{(\theta-u)}{\Phi(\theta)}e^{-\Phi(\theta)a} W^{(u)}(a-y)\\
&\hspace{-1.15cm}+\frac{(\theta-u)}{\Phi(\theta)}e^{-\Phi(\theta)a}\int_0^{\infty} e^{-\Phi(\theta)z} W^{(u)\prime}(z+a-y)dz\\
&\hspace{-1.15cm}-\frac{W^{(u)}(a-y)}{W^{(u)\prime}(a)}\Big[(\theta-u)e^{-\Phi(\theta)a}\int_0^{\infty} e^{-\Phi(\theta)z} W^{(u)\prime}(z+a)dz\Big].
\end{split}
\end{equation*}
Define $\Gamma(x,r)=\int_x^{\infty}\frac{z}{r}\mathbb{P}\{X_r\in dz\}.$ Following the above, we have from (\ref{eq:proof1}) that
\begin{align*}
&\hspace{-0.35cm}\int_0^{\infty} dr e^{-\theta r} \mathbb{E}_{\vert y}\Big\{ e^{-u\tau_a^+}\mathbf{1}_{\{\tau_a^+<\infty\}}\mathbb{P}_{\vert Y_{\tau_a^+}}\big\{\tau_{a-\epsilon}^-\leq r\big\}\Big\}\\
&\hspace{1.5cm} = \frac{1}{\theta} \mathbb{E}_{\vert y}\Big\{ e^{-u\tau_a^+}\mathbf{1}_{\{\tau_a^+<\infty\}}\mathbb{P}_{\vert Y_{\tau_a^+}}\big\{\mathbf{e}_{\theta} \geq \tau_{a-\epsilon}^-\big\}\Big\} \nonumber \\
&\hspace{2.5cm}=\frac{(\theta-u)}{\theta\Phi(\theta)} e^{-\Phi(\theta)\epsilon}W^{(u)}(a-y) \\
&\hspace{3cm}+ \frac{(\theta-u)}{\theta\Phi(\theta)}e^{-\Phi(\theta)\epsilon}\int_0^{\infty} e^{-\Phi(\theta)z}W^{(u)\prime}(z+a-y)dz  \nonumber\\
&\hspace{3cm}-\frac{W^{(u)}(a-y)}{\theta W^{(u)\prime}(a)}\Big[(\theta-u)e^{-\Phi(\theta)\epsilon}\int_0^{\infty} e^{-\Phi(\theta)z} W^{(u)\prime}(z+a)dz  \Big].
\end{align*}
Next, recall following (\ref{eq:pers1})-(\ref{eq:pers2}), (\ref{eq:fptabove}) and the Kendall's identity (\ref{eq:kendall}) that
\begin{align*}\label{eq:identity3}
\Big(\frac{1}{\Phi(\theta)} - \frac{u}{\theta \Phi(\theta)}\Big)e^{-\Phi(\theta)x}&=
\int_0^{\infty} dr e^{-\theta r} \Big(\Gamma(x,r) -u \int_0^r \Gamma(x,t) dt\Big)\\
\int_0^{\infty} \Big(1-\frac{u}{\theta}\Big)e^{-\Phi(\theta)(z+\epsilon)}W^{(u)\prime}(z+a)dz &=
\int_0^{\infty} dr e^{-\theta r}\Big(\Lambda_{\epsilon}^{(u)}(a,r)-u\int_0^r \Lambda_{\epsilon}^{(u)}(a,t)dt\Big).\nonumber
\end{align*}
Moreover, by applying integration by part we have after some calculations that
\begin{equation}\label{eq:identity4}
\int_0^{\infty} dz W^{(u)\prime}(z+x) \Gamma(z+\epsilon,t)=\Omega_{\epsilon}^{(u)}(x,t)-W^{(u)}(x)\Gamma(\epsilon,t).
\end{equation}
The claim in (\ref{eq:prop1}) is established following the above and by Tonelli and Laplace inversion (noting that both sides of (\ref{eq:proof1}) is right-continuous in $r$) to (\ref{eq:proof1}). \exit
\end{proof}
\begin{prop}
For given $u,r,\epsilon \geq 0$ and $a>0$, we have for any $y\geq0$ that
\begin{equation}
\begin{split}
&\mathbb{E}_{\vert y}\Big\{e^{-u\tau_a^+} \mathbf{1}_{\{\tau_a^+<\infty\}} \mathbb{E}_{\vert Y_{\tau_a^+}}\big\{e^{-u\tau_{a-\epsilon}^-}\mathbf{1}_{\{\tau_{a-\epsilon}^-\leq r\}}\big\}\Big\}\\
&\hspace{3.5cm}=e^{-ur}\Big(\Omega_{\epsilon}^{(u)}(a-y,r)-\frac{W^{(u)}(a-y)}{W^{(u)\prime}(a)}\Lambda_{\epsilon}^{(u)}(a,r)\Big).\label{eq:prop2}
\end{split}
\end{equation}
\end{prop}
\begin{proof}
On recalling (\ref{eq:fptabove}), we have by Tonelli, Lemma \ref{lem:identity1} and Corollary \ref{cor:corid1},
\begin{equation}\label{eq:proof2}
\begin{split}
&\int_0^{\infty} dr e^{-\theta r} \mathbb{E}_{\vert y}\Big\{e^{-u\tau_a^+}\mathbf{1}_{\{\tau_a^+<\infty\}} 
\mathbb{E}_{\vert Y_{\tau_a^+}}\big\{e^{-u\tau_{a-\epsilon}^-}\mathbf{1}_{\{\tau_{a-\epsilon}^- \leq r\}}\big\}\Big\}\\
&\hspace{2.5cm}= \frac{1}{\theta}\mathbb{E}_{\vert y}\Big\{e^{-u\tau_a^+}\mathbf{1}_{\{\tau_a^+<\infty\}} 
\mathbb{E}_{\vert Y_{\tau_a^+}}\big\{e^{-u\tau_{a-\epsilon}^-}\mathbf{1}_{\{\mathbf{e}_{\theta} \geq \tau_{a-\epsilon}^-\}}\big\}\Big\} \\
&\hspace{2.5cm}=\frac{1}{\theta}e^{\Phi(\theta+u)(a-\epsilon)}\mathbb{E}_{\vert y}\Big\{e^{-u\tau_a^+ -\Phi(\theta +u)Y_{\tau_a^+}}\mathbf{1}_{\{\tau_a^+<\infty\}}\Big\}.
\end{split}
\end{equation}
From Corollary \ref{cor:corid1}, the expectation on the right hand side is given by
\begin{equation*}
\begin{split}
\mathbb{E}_{\vert y}\Big\{e^{-u\tau_a^+ -\Phi(\theta +u)Y_{\tau_a^+}}\mathbf{1}_{\{\tau_a^+<\infty\}}\Big\}
=&\frac{\theta}{\Phi(\theta+u)}e^{-\Phi(\theta+u)a}W^{(u)}(a-y)\\
&\hspace{-2cm}+\frac{\theta}{\Phi(\theta+u)}e^{-\Phi(\theta+u)a}\int_0^{\infty}e^{-\Phi(\theta+u)z}W^{(u)\prime}(z+a-y)dz\\
&\hspace{-2cm}-\frac{W^{(u)}(a-y)}{W^{(u)\prime}(a)} \theta e^{-\Phi(\theta+u)a}\int_0^{\infty}e^{-\Phi(\theta+u)z}W^{(u)\prime}(z+a)dz.
\end{split}
\end{equation*}
Following the above, we have from the Laplace transform (\ref{eq:proof2}) that
\begin{align}\label{eq:proof3}
&\int_0^{\infty} dr e^{-\theta r} \mathbb{E}_{\vert y}\Big\{e^{-u\tau_a^+}\mathbf{1}_{\{\tau_a^+<\infty\}} 
\mathbb{E}_{\vert Y_{\tau_a^+}}\big\{e^{-u\tau_{a-\epsilon}^-}\mathbf{1}_{\{\tau_{a-\epsilon}^- \leq r\}}\big\}\Big\} \nonumber \\
&\hspace{1cm}=\frac{1}{\Phi(\theta+u)}e^{-\Phi(\theta+u)\epsilon}W^{(u)}(a-y) \nonumber\\
&\hspace{2cm}+\frac{1}{\Phi(\theta+u)}e^{-\Phi(\theta+u)\epsilon}\int_0^{\infty}e^{-\Phi(\theta+u)z} W^{(u)\prime}(z+a-y)dz\\
&\hspace{2cm}-\frac{W^{(u)}(a-y)}{W^{(u)\prime}(a)}e^{-\Phi(\theta+u)\epsilon}
\int_0^{\infty}e^{-\Phi(\theta+u)z} W^{(u)\prime}(z+a)dz. \nonumber
\end{align}
Moreover, following (\ref{eq:fptabove}), we have by applying Kendall's identity and (\ref{eq:pers1})
\begin{equation*}
\begin{split}
\frac{1}{\Phi(\theta+u)}e^{-\Phi(\theta+u)x}=&\int_0^{\infty} dr e^{-\theta r} e^{-ur}\Gamma(x,r)\\
\int_0^{\infty} e^{-\Phi(\theta+u)(z+\epsilon)} W^{(u)\prime}(z+a)dz=&\int_0^{\infty} dr e^{-\theta r} e^{-ur} \Lambda_{\epsilon}^{(u)}(a,r).
\end{split}
\end{equation*}
The claim (\ref{eq:prop2}) is justified using the above and (\ref{eq:identity4}) and by Tonelli and Laplace inversion of (\ref{eq:proof3}) - noting that both sides of (\ref{eq:proof3}) is right-continuous in $r$. \exit\\
\end{proof}

From the above two propositions, we have following (\ref{eq:derivation1}) and (\ref{eq:avram}) that
\begin{align}
\mathbb{E}_{\vert y}\big\{e^{-u\tau_r^{\epsilon}}\mathbf{1}_{\{\tau_r^{\epsilon}<\infty\}}\big\}&=e^{-ur}\Big[1+u\overline{W}^{(u)}(a-y) - u \frac{W^{(u)}(a)}{W^{(u)\prime}(a)}W^{(u)}(a-y)\Big] \nonumber\\
&-e^{-ur}\Big[\Omega_{\epsilon}^{(u)}(a-y,r)-u\int_0^r \Omega_{\epsilon}^{(u)}(a-y,t) dt  \label{eq:derivation6}\\
&-\frac{W^{(u)}(a-y)}{W^{(u)\prime}(a)}\Big(\Lambda_{\epsilon}^{(u)}(a,r) - u \int_0^r \Lambda_{\epsilon}^{(u)}(a,t) dt\Big)\Big]  \nonumber\\
&\hspace{-2cm}+ e^{-ur}\Big(\Omega_{\epsilon}^{(u)}(a-y,r) - \frac{W^{(u)}(a-y)}{W^{(u)\prime}(a)}\Lambda_{\epsilon}^{(u)}(a,r)\Big)\mathbb{E}_{\vert a-\epsilon}\big\{e^{-u\tau_r^{\epsilon}}\mathbf{1}_{\{\tau_r^{\epsilon}<\infty\}}\big\}.  \nonumber
\end{align}

We arrive at our claim (\ref{eq:main}) once the expectation on the right hand side is found. For this purpose, set $y=a-\epsilon$ on both sides of the above equation to get
\begin{align}
&\Big[1-e^{-ur}\Big(\Omega_{\epsilon}^{(u)}(\epsilon,r) - \frac{W^{(u)}(\epsilon)}{W^{(u)\prime}(a)}\Lambda_{\epsilon}^{(u)}(a,r) \Big)\Big]\mathbb{E}_{\vert a-\epsilon}\big\{e^{-u\tau_r^{\epsilon}}\mathbf{1}_{\{\tau_r^{\epsilon}<\infty\}}\big\}   \nonumber\\
&\hspace{0.5cm}=e^{-ur}\Big[1+u\overline{W}^{(u)}(\epsilon) - u \frac{W^{(u)}(a)}{W^{(u)\prime}(a)}W^{(u)}(\epsilon) - \Omega_{\epsilon}^{(u)}(\epsilon,r) + u\int_0^r \Omega_{\epsilon}^{(u)}(\epsilon,t)dt   \nonumber\\
&\hspace{4cm}+\frac{W^{(u)}(\epsilon)}{W^{(u)\prime}(a)}\Big(\Lambda_{\epsilon}^{(u)}(a,r) - u \int_0^r \Lambda_{\epsilon}^{(u)}(a,t) dt\Big)\Big].  \label{eq:derivation5}
\end{align}
However, on account of (\ref{eq:pers3}), we can rewrite the terms $\Omega_{\epsilon}^{(u)}(\epsilon,t)$ as follows
\begin{equation*}
 \Omega_{\epsilon}^{(u)}(\epsilon,t) = e^{ut} - \int_0^{\epsilon} W^{(u)}(z)\frac{z}{t}\mathbb{P}\{X_t\in dz\},
\end{equation*}
from which the equation (\ref{eq:derivation5}) simplifies further after some calculations to
\begin{equation*}
\begin{split}
&\Big[\int_0^{\epsilon} W^{(u)}(z)\frac{z}{r}\mathbb{P}\{X_r\in dz\} + \frac{W^{(u)}(\epsilon)}{W^{(u)\prime}(a)}\Lambda_{\epsilon}^{(u)}(a,r)\Big]\mathbb{E}_{\vert a-\epsilon}\big\{e^{-u\tau_r^{\epsilon}}\mathbf{1}_{\{\tau_r^{\epsilon}<\infty\}}\big\}\\
&\hspace{0cm}=\int_0^{\epsilon} W^{(u)}(z)\frac{z}{r}\mathbb{P}\{X_r\in dz\} + \frac{W^{(u)}(\epsilon)}{W^{(u)\prime}(a)}\Lambda_{\epsilon}^{(u)}(a,r)+u\Big[\overline{W}^{(u)}(\epsilon) - \frac{W^{(u)}(a)}{W^{(u)\prime}(a)}W^{(u)}(\epsilon) \\
&\hspace{3cm}-\int_0^rdt \Big(\int_0^{\epsilon}W^{(u)}(z)\frac{z}{t}\mathbb{P}\{X_t\in dz\}+\frac{W^{(u)}(\epsilon)}{W^{(u)\prime}(a)}\Lambda_{\epsilon}^{(u)}(a,t)\Big)\Big],
\end{split}
\end{equation*}
or equivalently, we obtain after dividing both sides of the equation by $ W^{(u)}(\epsilon)$,
\begin{equation*}
\begin{split}
&\mathbb{E}_{\vert a-\epsilon}\big\{e^{-u\tau_r^{\epsilon}}\mathbf{1}_{\{\tau_r^{\epsilon}<\infty\}}\big\}\\
=&
1+u\frac{\Big[\frac{\overline{W}^{(u)}(\epsilon)}{W^{(u)}(\epsilon)} - \frac{W^{(u)}(a)}{W^{(u)\prime}(a)} - \int_0^r dt \Big(\int_0^{\epsilon} \frac{W^{(u)}(z)}{ W^{(u)}(\epsilon)}\frac{z}{t}\mathbb{P}\{X_t\in dz\} + \frac{\Lambda_{\epsilon}^{(u)}(a,t)}{W^{(u)\prime}(a)}\Big)\Big]}{\Big(\int_0^{\epsilon} \frac{W^{(u)}(z)}{ W^{(u)}(\epsilon)}\frac{z}{r}\mathbb{P}\{X_r\in dz\} + \frac{\Lambda_{\epsilon}^{(u)}(a,r)}{W^{(u)\prime}(a)}\Big)}.
\end{split}
\end{equation*}

Using this result and putting it back in the equation (\ref{eq:derivation6}), we arrive at
\begin{align}\label{eq:proof4}
e^{ur}\mathbb{E}_{\vert y}\big\{e^{-u\tau_r^{\epsilon}}\mathbf{1}_{\{\tau_r^{\epsilon}<\infty\}}\big\}&=1+u\overline{W}^{(u)}(a-y) - u \frac{W^{(u)}(a)}{W^{(u)\prime}(a)}W^{(u)}(a-y)  \nonumber\\
&+u\int_0^r dt \Big(\Omega_{\epsilon}^{(u)}(a-y,t)-\frac{W^{(u)}(a-y)}{W^{(u)\prime}(a)}\Lambda_{\epsilon}^{(u)}(a,t)\Big)  \nonumber \\
&\hspace{-2cm}+u\frac{\Big[\frac{\overline{W}^{(u)}(\epsilon)}{W^{(u)}(\epsilon)} - \frac{W^{(u)}(a)}{W^{(u)\prime}(a)} - \int_0^r dt \Big(\int_0^{\epsilon} \frac{W^{(u)}(z)}{ W^{(u)}(\epsilon)}\frac{z}{t}\mathbb{P}\{X_t\in dz\} + \frac{\Lambda_{\epsilon}^{(u)}(a,t)}{W^{(u)\prime}(a)}\Big)\Big]}{\Big(\int_0^{\epsilon} \frac{W^{(u)}(z)}{ W^{(u)}(\epsilon)}\frac{z}{r}\mathbb{P}\{X_r\in dz\} + \frac{\Lambda_{\epsilon}^{(u)}(a,r)}{W^{(u)\prime}(a)}\Big)}  \nonumber\\
&\times \Big(\Omega_{\epsilon}^{(u)}(a-y,r) - \frac{W^{(u)}(a-y)}{W^{(u)\prime}(a)}\Lambda_{\epsilon}^{(u)}(a,r)\Big).
\end{align}
We now want to compute the limit as $\epsilon\downarrow 0$ of (\ref{eq:proof4}). In order to do this, recall by the spatial homogeneity that $\mathbb{P}_{\vert y}\{\tau_r^{\epsilon}\leq t\}=\mathbb{P}_{\vert y+\epsilon}\{\tau_r\leq t\}$ and therefore by the right-continuity of the map $y\rightarrow \mathbb{P}_{\vert y}\{\tau_r\leq t\}$, we have $\mathbb{P}_{\vert y}\{\tau_r\leq t\}=\lim_{\epsilon\downarrow 0}\mathbb{P}_{\vert y}\{\tau_r^{\epsilon}\leq t\}$. Hence, by weak convergence theorem the Laplace transform of $\mathbb{P}_{\vert y}\{\tau_r^{\epsilon}\leq t\}$ converges as $\epsilon\downarrow 0$ to that of $\mathbb{P}_{\vert y}\{\tau_r\leq t\}$, i.e., $\lim_{\epsilon\downarrow 0}\mathbb{E}_{\vert y}\big\{e^{-u\tau_r^{\epsilon}}\mathbf{1}_{\{\tau_r^{\epsilon}<\infty\}}\big\}=\mathbb{E}_{\vert y}\big\{e^{-u\tau_r}\mathbf{1}_{\{\tau_r<\infty\}}\big\}$. 

We consider two cases: $W^{(u)}(0+)>0$ ($X$ has paths of bounded variation) and $W^{(u)}(0+)=0$ ($X$ has unbounded variation). For the case $W^{(u)}(0+)>0$,
\begin{align}
e^{ur}\mathbb{E}_{\vert y}\big\{e^{-u\tau_r}\mathbf{1}_{\{\tau_r<\infty\}}\big\}&=1+u\overline{W}^{(u)}(a-y) - u \frac{W^{(u)}(a)}{W^{(u)\prime}(a)}W^{(u)}(a-y)  \nonumber\\
&+u\int_0^r dt \Big(\Omega^{(u)}(a-y,t)-\frac{W^{(u)}(a-y)}{W^{(u)\prime}(a)}\Lambda^{(u)}(a,t)\Big)  \label{eq:limit1} \\
&\hspace{-3.25cm} -u\Big(\frac{W^{(u)}(a)}{\Lambda^{(u)}(a,r)} + \int_0^r \frac{\Lambda^{(u)}(a,t)}{\Lambda^{(u)}(a,r)} dt \Big)  \Big(\Omega^{(u)}(a-y,r) - \frac{W^{(u)}(a-y)}{W^{(u)\prime}(a)}\Lambda^{(u)}(a,r)\Big), \nonumber
\end{align}
which after some further calculations simplifies to the main result (\ref{eq:main}). 

For the case $W^{(u)}(0+)=0$, we have after applying integration by parts that
\begin{equation*}
\int_0^{\epsilon} \frac{W^{(u)}(z)}{W^{(u)}(\epsilon)} \frac{z}{r} \mathbb{P}\{X_r\in dz\}
=\int_0^{\epsilon} \frac{z}{r}\mathbb{P}\{X_r\in dz\}  - \int_0^{\epsilon} dz \frac{W^{(u)\prime}(z)}{W^{(u)}(\epsilon)}\int_0^z \frac{w}{r} \mathbb{P}\{X_r \in dw\}.
\end{equation*}
Therefore, by employing l'H\^opital rule we obtain
\begin{equation*}
\begin{split}
\lim_{\epsilon\downarrow 0} \int_0^{\epsilon} \frac{W^{(u)}(z)}{W^{(u)}(\epsilon)} \frac{z}{r} \mathbb{P}\{X_r\in dz\}=
\lim_{\epsilon\downarrow 0} \frac{W^{(u)\prime}(\epsilon)\int_0^{\epsilon}\frac{z}{r}\mathbb{P}\{X_r\in dz\}}{W^{(u)\prime}(\epsilon)}=0.
\end{split}
\end{equation*}
The claim is established once we show that $\lim_{\epsilon\downarrow 0} \frac{\overline{W}^{(u)}(\epsilon)}{W^{(u)}(\epsilon)} =\lim_{\epsilon\downarrow 0}\frac{W^{(u)}(\epsilon)}{W^{(u)\prime}(\epsilon)}=0$. This turns out to be the case when $X$ has paths of unbounded variation since $W^{(u)\prime}(0+)=2/\sigma^2$ if $\sigma\neq 0$ and is equal to $\infty$ if $\sigma=0$. See for instance Lemma 4.4. in Kyprianou and Surya \cite{Kyprianou2007}. On account of these results, we arrive at the identity (\ref{eq:limit1}), which after some further calculations simplifies to the main result (\ref{eq:main}). \exit

We have shown that (\ref{eq:main}) holds for $z\leq a$. We now prove that (\ref{eq:main}) holds for $z>a$. For this purpose, recall that under measure $\mathbb{P}_{\vert y}$, with $y>a$, $\tau_a^+=0$ a.s. On account of the fact $W^{(u)}(x)=0$ for $x<0$, we have from (\ref{eq:prop1}) and (\ref{eq:prop2}) for $\epsilon=0$ and $y>a$,
\begin{align}
\mathbb{P}_{\vert y}\{\tau_{a}^-\leq r\}&=\Omega^{(u)}(a-y,r)-u\int_0^r \Omega^{(u)}(a-y,t)dt.   \label{eq:ID1}\\
\mathbb{E}_{\vert y}\big\{e^{-u\tau_{a}^-}\mathbf{1}_{\{\tau_{a}^-\leq r\}}\big\}&=e^{-ur} \Omega^{(u)}(a-y,r). \label{eq:ID2}
\end{align}
These identities can be proved by Kendall's identity, Tonelli, (\ref{eq:fptabove}) and Laplace inversion taking account for $x<0$, $\int_0^{\infty} e^{-\theta t} \Omega^{(u)}(x,t)dt=e^{\Phi(\theta)x}/(\theta-u)$, $\theta>u$. Indeed,
\begin{align*}
\int_0^{\infty} e^{-\theta t} \Omega^{(u)}(x,t)dt &= \int_0^{\infty} e^{-\theta t} \int_0^{\infty} W^{(u)}(z+x)\frac{z}{t}\mathbb{P}\{X_t\in dz\}dt\\
&=\int_0^{\infty} e^{-\theta t} \int_0^{\infty} W^{(u)}(z+x)\mathbb{P}\{T_z^+\in dt\} dz\\
&= \int_0^{\infty} W^{(u)}(z+x) \int_0^{\infty} e^{-\theta t} \mathbb{P}\{T_z^+\in dt\} dz\\
&=\int_0^{\infty} W^{(u)}(z+x) e^{-\Phi(\theta)z} dz\\
&=e^{\Phi(\theta) x} \int_x^{\infty} e^{-\Phi(\theta)z} W^{(u)}(z) dz. \quad \exit
\end{align*}

Starting from eqn. (\ref{eq:mainyneg}) with $\epsilon=0$, we obtain following identities (\ref{eq:ID1})-(\ref{eq:ID2}),
\begin{equation}\label{eq:maina0}
\begin{split}
\mathbb{E}_{\vert y}\big\{e^{-u\tau_r}\mathbf{1}_{\{\tau_r<\infty\}}\big\}&=e^{-ur}\Big[1+u\int_0^r \Omega^{(u)}(a-y,t)dt -\Omega^{(u)}(a-y,r)   \Big] \\
&+ e^{-ur} \Omega^{(u)}(a-y,r) \mathbb{E}_{\vert a}\big\{e^{-u\tau_r}\mathbf{1}_{\{\tau_r<\infty\}}\big\}.
\end{split}
\end{equation}
The expression for $\mathbb{E}_{\vert a}\big\{e^{-u\tau_r}\mathbf{1}_{\{\tau_r<\infty\}}\big\}$ is given by setting $z=a$ in (\ref{eq:main}):
\begin{equation}\label{eq:maina}
\mathbb{E}_{\vert a}\big\{e^{-u\tau_r}\mathbf{1}_{\{\tau_r<\infty\}}\big\}=1-u\Big(\frac{W^{(u)}(a)}{\Lambda^{(u)}(a,r)}+\int_0^r \frac{\Lambda^{(u)}(a,t)}{\Lambda^{(u)}(a,r)}dt\Big),
\end{equation}
 where we have used in the calculation above the fact that $\Omega^{(u)}(0,t)=e^{ut}$, see eqn. (\ref{eq:pers3}). By inserting (\ref{eq:maina}) in (\ref{eq:maina0}) we obtain after some further calculations that
 \begin{align*}
 \mathbb{E}_{\vert y}\big\{e^{-u\tau_r}\mathbf{1}_{\{\tau_r<\infty\}}\big\}&= e^{-ur}\Big\{1+u\Big[-\frac{\Omega^{(u)}(a-y,r)}{\Lambda^{(u)}(a,r)}W^{(u)}(a) \\
 &\hspace{2cm}+\int_0^r \Big(\Omega^{(u)}(a-y,t)-\frac{\Omega^{(u)}(a-y,r)}{\Lambda^{(u)}(a,r)}\Lambda^{(u)}(a,t)\Big)dt\Big]\Big\},
 \end{align*}
 which corresponds to (\ref{eq:main}) for $z>a$ showing that (\ref{eq:main}) holds for any $z\geq 0$. \exit

\subsection{Proof of Proposition \ref{prop:maincor}}
Applying Esscher transform of measure (\ref{eq:esscher}) to the result (\ref{eq:main}), we have 
\begin{align}
\mathbb{E}_{y,x}\big\{e^{-u\tau_r + \nu X_{\tau_r}} \mathbf{1}_{\{\tau_r<\infty\}}\big\}=&
e^{\nu x}\mathbb{E}_{y,x}\big\{e^{-p\tau_r} e^{\nu(X_{\tau_r}-x)-\psi(\nu)\tau_r} \mathbf{1}_{\{\tau_r<\infty\}}\big\}  \nonumber\\
=&e^{\nu x}\mathbb{E}_{y,x}^{\nu}\big\{e^{-p\tau_r}  \mathbf{1}_{\{\tau_r<\infty\}}\big\}, \label{eq:hasil2}
\end{align}
where we have defined $p=u-\psi(\nu)$. Under the new measure $\mathbb{P}^{\nu}$,
 \begin{align*}
 \mathbb{E}_{y,x}^{\nu}\big\{e^{-p\tau_r}\mathbf{1}_{\{\tau_r<\infty\}}\big\}&=e^{-pr}\Big\{ 1+ p \Big[\overline{W}_{\nu}^{(p)}(a+x-y) -\frac{\Omega_{\nu}^{(p)}(a+x-y,r)}{\Lambda_{\nu}^{(p)}(a,r)}W_{\nu}^{(p)}(a) \nonumber\\
&\hspace{-2cm}+ \int_0^r \Big(\Omega_{\nu}^{(p)}(a+x-y,t)-\frac{\Omega_{\nu}^{(p)}(a+x-y,r)}{\Lambda_{\nu}^{(p)}(a,r)}\Lambda_{\nu}^{(p)}(a,t)\Big) dt \Big]  \Big\},
 \end{align*}
following which and the equation (\ref{eq:hasil2}) our claim in (\ref{eq:main2}) is established. \exit 

\section{Conclusions}\label{sec:conclusions}
We have presented some new results concerning Parisian ruin problem under L\'evy insurance risk process, where ruin is announced when the risk process has gone below a certain level from the last record maximum of the process, also known as the drawdown, for a fixed consecutive period of time. They further extend the existing results on Parisian ruin below a fixed level of the risk process. Using recent developments on fluctuation and excursion theory of the drawdown of the L\'evy risk process,  the law of ruin-time and the position at ruin was given in terms of their joint Laplace transforms. Identities are presented semi-explicitly in terms of the scale function and the law of the L\'evy process. The results can be used to calculate some quantities of interest in finance and insurance as discussed in the introduction.

\section*{Acknowledgement}
The author would like to thank a number of anonymous referees and associate editors for their useful suggestions and comments that improved the presentation of this paper. This paper was completed during the time the author visited the Hugo Steinhaus Center of Mathematics at Wroc\l aw University of Science and Technology in Poland. The author acknowledges the support and hospitality provided by the Center. He thanks to Professor Zbigniew Palmowski for the invitation, and for some suggestions over the work discussed during the MATRIX Mathematics of Risk Workshop in Melbourne organized by Professors Konstantin Borovkov, Alexander Novikov and Kais Hamza to whom the author also like to thanks for the invitation. This research is financially supported by Victoria University PBRF Research Grants \# 212885 and \# 214168 for which the author is grateful.


\begin{thebibliography}{99.}%
\bibitem{Agarwal} Agarwal, V., Daniel, N., Naik, N.: Role of managerial incentives and discretion in hedge fund performance. J. Finance \textbf{64}, 2221-2256 (2009)

\bibitem{Avram} Avram, F., Kyprianou, A.E., Pistorius, M.R.: Exit problems for spectrally negative L\'evy processes and applications to (Canadized) Russian Options. Ann. Appl. Probab. \textbf{14}, 215-238 (2004)

\bibitem{Bertoin} Bertoin, J.: L\'evy Processes. Cambridge University Press, Cambridge (1996) 

\bibitem{Broadie} Broadie, M., Chernov, M., Sundaresan, S.: Optimal debt and equity values in the presence of Chapter 7
and Chapter 11. J. Finance \textbf{LXII}, 1341-1377 (2007)

\bibitem{Chan} Chan, T., Kyprianou, A.E. and Savov, M.: Smoothness of scale functions for spectrally negative L\'evy processes. Probab. Theory Rel. \textbf{150}, 691-708 (2011)

\bibitem{Chesney} Chesney, M., Jeanblanc-Picqu\'e, M., Yor, M.: Brownian excursions
and Parisian barrier options. Adv. Appl. Probab. \textbf{29}, 165-184 (1997)

\bibitem{Czarna} Czarna, I., Palmowski, Z.: Ruin probability with Parisian delay for a spectrally negative L\'evy process. J. Appl. Probab. \textbf{48}, 984-1002 (2011)

\bibitem{Dassios2010} Dassios, A., Wu, S.: Perturbed Brownian motion and its application to
Parisian option pricing. Finance Stoch. \textbf{14}, 473-494 (2010)

\bibitem{Francois} Francois, P., Morellec, E.: Capital structure and asset prices: Some effects of bankruptcy
procedures. J. Business \textbf{77}, 387-411 (2004)

\bibitem{Goetzmann} Goetzmann, W.N., Ingersoll Jr., J.E., Ross, S.A.: High-water marks and hedge fund management contracts. J. Finance \textbf{58}, 1685–1717 (2003)

\bibitem{Kuznetzov} Kusnetzov, A., Kyprianou, A.E., Rivero, V.: The Theory of Scale Functions for Spectrally
Negative L\'evy Processes, L\'evy Matters II.  Springer Lecture Notes in Mathematics (2013)

\bibitem{Kyprianou2007} Kyprianou, A.E., Surya, B.A.: Principles of smooth and continuous fit in the determination of endogenous bankruptcy levels. Finance Stoch. \textbf{11}, 131-152 (2007)

\bibitem{Kyprianou} Kyprianou, A.E.: Introductory Lectures on Fluctuations of L\'evy Processes with Applications. Springer, Berlin (2006)

\bibitem{Lambert} Lambert, A.: Completely asymmetric L\'evy processes confined in a finite interval. Ann. Inst. Henri Poincar\'e \textbf{2},  251-274 (2000)

\bibitem{Landriault} Landriault, D., Renaud, J-F., Zhou, X.: An insurance risk model with Parisian implementation delays. Methodol. Comput. Appl. Probab. \textbf{16}, 583-607 (2014)

\bibitem{Loeffen2017} Loeffen, R., Palmowski, Z., Surya, B.A.: Discounted penalty function at Parisian ruin for L\'evy insurance risk process. Insur. Math. Econ. (2017)

\bibitem{Loeffen} Loeffen, R., Czarna, I., Palmowski, Z.: Parisian ruin probability for spectrally
negative L\'evy processes. Bernoulli \textbf{19}, 599-609 (2013)

\bibitem{Mijatovic} Mijatovi\'c, A., Pistorius, M.R.: On the drawdown of completely assymetric L\'evy process. Stoc. Proc. Appl. \textbf{122}, 3812-3836 (2012)

\bibitem{Palmowski2018} Palmowski, Z., Tumilewicz, J.: Pricing insurance drawdown-type contracts with underlying L\'evy assets.  Insur. Math. Econ. \textbf{79}, 1-14 (2018)

\bibitem{Surya} Surya, B.A.: Evaluating scale function of spectrally negative L\'evy processes. J. Appl. Probab. \textbf{45}, 135-149 (2008)

\bibitem{Zhang} Zhang, H., Leung, T., Hadjiliadis, O.: Stochastic modeling and fair valuation of drawdown insurance.  Insur. Math. Econ. \textbf{53}, 840-850 (2013)
%
\end{thebibliography}
\end{document}